\documentclass[12pt,a4paper,leqno]{amsart}
\usepackage[utf8]{inputenc}
\usepackage{amssymb,color,tikz}
\usepackage{amsmath,amscd,amsfonts,amsthm,verbatim}
\usepackage[all]{xy}

\newcommand{\CC}{{\mathbb C}}
\newcommand{\bH}{{\mathbb H}}
\newcommand{\cE}{{\mathcal E}}
\newcommand{\cS}{{\mathcal S}}
\newcommand{\cF}{{\mathcal F}}
\newcommand{\cG}{{\mathcal G}}
\newcommand{\cP}{{\mathcal P}}
\newcommand{\cV}{{\mathcal V}}
\newcommand{\cZ}{{\mathcal Z}}

\newcommand{\cO}{{\mathcal O}}

\newcommand{\cU}{{\mathcal U}}
\newcommand {\PP}{\mathbb{P}}

\newcommand{\cL}{{\mathcal L}}

\newcommand {\ZZ}{\mathbb{Z}}
\newcommand {\NN}{\mathbb{N}}
\newcommand{\deb}{\overline{\partial}}

\DeclareMathOperator{\Gm}{\mathbb{G}_{\text{m}}}
\DeclareMathOperator{\Syz}{Syz}

\DeclareMathOperator{\EExt}{{\cE}xt}

\DeclareMathOperator{\Spl}{Spl}
\DeclareMathOperator{\SSpl}{{\cS}pl}
\DeclareMathOperator{\Pic}{Pic}
\DeclareMathOperator{\End}{End}
\DeclareMathOperator{\ext}{Ext}
\DeclareMathOperator{\tr}{tr}

\DeclareMathOperator{\rank}{rank}
\DeclareMathOperator{\id}{id}

\DeclareMathOperator{\Ext}{Ext}

\DeclareMathOperator{\Aut}{Aut}

\newtheorem{theorem}{Theorem}[section]
\newtheorem{lemma}[theorem]{Lemma}
\newtheorem{question}[theorem]{Question}
\newtheorem{proposition}[theorem]{Proposition}
\newtheorem{corollary}[theorem]{Corollary}
 \theoremstyle{definition}
\newtheorem{definition}[theorem]{Definition}
\newtheorem{example}[theorem]{Example}
\newtheorem{remark}[theorem]{Remark}
\newtheorem{Notation}[theorem]{Notation}

\title[Lagrangian subspaces]{Lagrangian subspaces of the moduli space of simple sheaves on K3 surfaces}

\begin{document}

 \author[B.\ Fantechi]{Barbara Fantechi} 
 \address{SISSA
Via Bonomea 265, I-34136 Trieste, Italy}
  \email{fantechi@sissa.it, 
  ORCID 0000-0002-7109-6818}.
  \author[R.\ M.\ Mir\'o-Roig]{Rosa M.\ Mir\'o-Roig} 
  \address{Facultat de
  Matem\`atiques i Inform\`atica, Universitat de Barcelona, Gran Via des les
  Corts Catalanes 585, 08007 Barcelona, Spain} \email{miro@ub.edu, ORCID 0000-0003-1375-6547}

\thanks{The first author has been partially supported by PRIN 2017 Moduli theory and
birational classification.} 
\thanks{The second author has been partially supported by the grant PID2019-104844GB-I00}

\begin{abstract} Let $X$ be a  K3 surface and let $\Spl(r;c_1,c_2)$ be the moduli space of simple sheaves on $X$ of fixed rank $r$ and Chern classes $c_1$ and $c_2$.  Under suitable assumptions, to a pair $(F,W)$ (respectively, $(F,V)$) where $F\in \Spl(r;c_1,c_2)$ and $W\subset H^0(F)$ (resp.~$V^*\subset H^1(F^*)$) is a vector subspace, we  associate a simple syzygy bundle (resp.~extension bundle) on $X$.

We show that both syzygy bundles and extension bundles can be constructed in families and that the induced morphism to a different component of the moduli of simple sheaves is a locally closed embedding.

We show that this construction associates to every Lagrangian (resp.~isotropic) algebraic subspace of $\Spl(r;c_1,c_2)$ an induced Lagrangian (resp.~isotropic) algebraic subspace of a different component of the moduli of simple sheaves.

\end{abstract}

\maketitle

\tableofcontents

\section{Introduction}
A smooth complex variety (or more generally algebraic space) is called holomorphic symplectic if it has a nowhere degenerate closed holomorphic 2-form. In 1984, S.~Mukai 
showed the existence of a natural non-degenerate closed holomorphic 2-form on the  moduli space $\Spl(r;c_1,c_2)$ of simple sheaves on a K3 surface with fixed rank and Chern classes \cite[Theorem 0.1]{Mu}, thus providing a whole series of examples of  holomorphic symplectic algebraic spaces.

In this paper, by K3 surface we always mean a smooth, connected, complex projective K3 surface $X$. We will present two natural ways to construct  isotropic and Lagrangian subspaces (i.e.  isotropic irreducible subspaces of maximal dimension) in the moduli space of simple bundles on a K3 surface, starting from other suitably chosen isotropic/Lagrangian subspaces. By subspace of an algebraic space we always mean "smooth, irreducible, locally closed algebraic subspace" throught this paper. As starting point for the construction, we remark that points in $\Pic(X)$ are Lagrangian subspaces.

The first method uses generalized syzygy bundles. Given a simple vector bundle $F$ on a K3 surface, and a vector subspace $W\subset H^0(F)$ of dimension $w\ge 2+rk (F)$, the {\em generalized syzygy} bundle $S$ associated to $(F,W)$ is the kernel of the evaluation map $W\otimes \cO_X \twoheadrightarrow F$, when the latter is surjective (see Definition \ref{def1}).

In \cite{FM},  for any K3 surface $X$ and for any $w$ we define the moduli space of generalized syzygy bundles $\cG_U^0$  to be the natural structure on pairs $(F,W)$ where $F\in \Spl(r;c_1,c_2)$ is a simple vector bundle such that $H^1(F)=0$  and $W\subset H^0(F)$ is a generating $w$-dimensional space. We show that any generalized syzygy bundle so obtained is simple of rank $r'=w-r$ and Chern classes $c_1'=-c_1$, and $c_2'=c_1^2-c_2$.

 We also show that the set maps $(F,W)\mapsto F$ and $(F,W)\mapsto S$ are induced by morphisms $\pi:\cG^0_U\to \Spl(r;c_1,c_2)$ and $\alpha:\cG^0_U\to \Spl(r';c_1',c_2')$.
The morphism $\pi:\cG^0_U\to \Spl(r;c_1,c_2)$ is smooth and \'etale locally quasiprojective over its image $U$. In fact, it is locally open in a Grassmann bundle.  Finally, we show that the morphism $\alpha :\cG^0_U \to \Spl(r';c_1',c_2')  $ is a locally closed embedding.

For every $\cL\subset \Spl(r;c_1,c_2)$ locally closed subspace, we let $\cG^0_{\cL}:=\pi^{-1}(\cL\cap U)$. The restriction of $\alpha $ to $\cG^0_{\cL}$ induces a locally closed embedding $\alpha :\cG_{\cL}^0\to \Spl(r';c_1',c_2')$ with image $\Syz^w_\cL$.
Our main result about moduli of generalized syzygy bundles is the following.

\begin{theorem}
\label{mainthm1}
Let $\cL\subset \Spl(r;c_1,c_2)$ be a smooth, irreducible, locally closed subspace and $w\ge 2 + r$ an integer. Then the  locally closed embedding $$\alpha : \cG^0_{\cL}\to \Spl(r';c_1',c_2') 
$$
is isotropic (respectively, Lagrangian) if and only if $\cL$ is isotropic (respectively, Lagrangian) in $\Spl(r;c_1,c_2)$. 
\end{theorem}

This allows us to determine when  $\Syz^w_\cL$ is an isotropic (respectively, Lagrangian) subspace of $\Spl(r';c_1',c_2')$.
We will show in Corollary \ref{FirstCoro} how this leads to the construction of infinitely many Lagrangian immersions just starting from points in $\Pic(X)$. 

\smallskip

The second construction is similar in spirit, as it associates a simple vector bundle to a pair consisting of a simple bundle $F$ on a K3 surface and a vector subspace of a cohomology group of $F$; it applies under very different assumptions on the vector bundle, namely when $H^0(F)=H^0(F^*)=0$ and $H^1(F)\ne 0$.

For such a vector bundle, let $V^*\subset H^1(F^*)$ be a  $v$-dimensional space; we can associate to it the {\em extension bundle} $E$ coming from the induced extension of $F$ by $V\otimes \cO_X$ 
(see Definition \ref{DefExt}).

We show that the {extension bundle}  $E$  is simple  of rank $\bar r=v+r$ and Chern classes $c_1$, and $c_2$.
We define the moduli space of extension bundles $\cP_{U_e}$  to be the natural moduli structure on pairs $(F,V)$ where $F\in \Spl(r;c_1,c_2)$ is a simple vector bundle such that $H^0(F)=H^0(F^*)=0$  and $V^*\subset H^1(F^*)$ is a  $v$-dimensional space.

We show that the set maps $(F,V)\mapsto F$ and $(F,V)\mapsto E$ are induced by morphisms $\pi:\cP_{U_e}\to \Spl(r;c_1,c_2)$ and $\beta:\cP_{U_e}\to \Spl(\bar r;c_1,c_2)$.

In analogy to the previous case, the morphism $\pi:\cP_{U_e} \to \Spl(r;c_1,c_2)$ is smooth and \'etale locally quasiprojective; in fact, it is locally open in a Grassmann bundle.
As a first result we prove  that the morphism $\beta :\cP_{U_e} \to \Spl(\bar r;c_1,c_2)  $ is a locally closed embedding.

For every $\cL\subset \Spl(r;c_1,c_2)$ locally closed subspace, we let $\cP_{\cL}:=\pi^{-1}(\cL\cap U_e)$. The restriction of $\beta $ to $\cP_{\cL}$ induces a locally closed embedding $\beta :\cP_{\cL}\to \Spl(\bar r;c_1,c_2) $ with image $\Ext^v_{\cL}$.
Our main result about moduli space of extension bundles is:

\begin{theorem}
\label{mainthm2}
Let $\cL\subset \Spl(r;c_1,c_2)$ be a smooth, irreducible, locally closed subspace and $v\ge 1$ an integer. Then the  locally closed embedding $$\beta : \cP_{\cL}\to \Spl(\overline{r};c_1,c_2)  
$$
is isotropic (respectively, Lagrangian) if and only if $\cL$ is isotropic (respectively, Lagrangian) in $\Spl(r;c_1,c_2)$. 
\end{theorem}

\vskip 2mm
The proof of our main results combines classical computations in vector bundle theory as well as the computation of tangent spaces using deformation theory. 

\vskip 2mm
\noindent{\em Organization of the paper.} In section 2, we set some notation, we recall some standard definitions as well as the constructions of generalized syzygy bundles and extension bundles; and for the sake of completeness we gather some known results on deformation theory and moduli spaces of simple sheaves.

In section 3, we collect the basic properties of generalized syzygy bundles as well as of extension bundles. We
also introduce the moduli spaces we are interested in, namely, the moduli space of syzygy bundles $\cG^0_\cL$ and the moduli space of extension bundles  $\cP_\cL$. 
We  check that if we restrict $F$ to move in a family of dimension half the dimension of the corresponding moduli space, then $S$ (resp.~$E$) also moves in a family of dimension half of the dimension of its moduli space. This is a key step in both of our main theorems, since it allows to go directly from the statement about isotropic subspaces to that about Lagrangian subspaces.

Sections 4 and 5 are the heart of the paper: they are integrally devoted to prove Theorems \ref{mainthm1} and \ref{mainthm2}.
In analogy with the results in \cite{FM} for generalized syzygy bundles, we show that the construction of extension bundles induces a morphism $\beta$ from $\cP_{U_e}$, the moduli space of extension bundles, to $\Spl(r+v;c_1,c_2)$ which is a locally closed embedding.

We then show that, for a fixed isotropic subspace $\cL\subset \Spl(r;c_1,c_2)$, the restriction of $\alpha$ to $\cG^0_\cL$ (resp.~of $\beta$ to $\cP_\cL$) is isotropic, by explicitly comparing the Mukai symplectic holomorphic 2-forms. The analogous result with isotropic replaced by Lagrangian follows immediately by the previous dimension count.

Finally, in section 6, we discuss the stability of the generalized syzygy bundles. We summarize what is known and we  pose a problem that naturally arises in our more general set up.

\vskip 2mm

\noindent {\bf Acknowledgement.} The first author is thankful for hospitality at Universitat de Barcelona and Chennai Mathematical Institute, as well as to Sissa for a sabbatical year, during which this research took place.

\section{Set up}

\subsection{Notation and conventions} Throughout this paper we work over the complex numbers $\CC$. A scheme, algebraic space, algebraic stack will always be assumed to be locally of finite type over $\CC$; a point will always be a $\CC$-valued (i.e. closed) point.

By $X$ is a {\em K3 surface} we mean that $X$ is a smooth projective connected complex surface with trivial canonical line bundle $\omega _{\cO_X}={\mathcal O_X}$ and irregularity $q=0$. Hence $X$ has $p_a=p_g=1$ and $\chi ({\mathcal O}_X)=2$. Basic examples of K3 surfaces are quartic surfaces in $\PP^3$ and complete intersections of a quadric and a cubic hypersurfaces in $\PP^4$.

We will fix a $K3$ surface $X$ and we denote by
$\Spl(r;c_1,c_2)$ the moduli space of rank $r$ simple sheaves $E$ on $X$ with fixed Chern classes $c_1$ and $c_2$.

\subsection{Moduli spaces of simple sheaves}\label{moduli_spaces}

In this paper we study moduli spaces of simple sheaves on the K3 surface $X$.  Here we collect for the reader's conveniences some basic information and we fix relevant notation. 

A coherent sheaf $E$ on a K3 surface $X$ is called  {\em simple} if the natural injection $$H^0(X,\cO_X) \longrightarrow \End_{\cO_X}(E)$$ is an isomorphism. In particular, because $X$ is smooth, projective and connected, $E$ being simple is equivalent to $\End_{\cO_X}(E)$ being one-dimensional, or to $\Aut(E)=\CC^*\id_E$.

The moduli space of simple sheaves on $X$ is a natural algebraic space structure on the set of isomorphism classes of simple sheaves. It is defined by the properties, that are expressed in terms of families.

A {\em family} of simple sheaves on $X$ parametrized by a scheme $B$ is a simple sheaf $\cF$ on $X\times B$, flat over $B$ and such that, for every $b\in B$, the sheaf $\cF_b:=\cF|_{X\times\{b\}}$ is simple. Since flatness and fibers commute with pullback, a family $\cF$ parametrized by $B$ induces a pullback family $(\id_X\times \varphi)^*\cF$ on $B'$ for every morphism $\varphi: B'\to B$. We remark that if $L\in \Pic(B)$, then $\tilde\cF:=\cF\otimes p_B^*L$ is also a family parametrized by $B$ and for every $b\in B$, the sheaves $\cF_b$ and $\tilde\cF_b$ are isomorphic.

The moduli space of simple sheaves on $X$ is a reduced algebraic space $\Spl(X)$ with the following properties:
\begin{enumerate}
    \item Its points are the set of isomorphism classes of simple sheaves on $X$;
    \item For every family $\cF$ of simple sheaves parametrized by a scheme $B$, the set theoretic map $B\to \Spl(X)$ given by $b\mapsto [\cF_b]$ is induced by a morphism;
    \item The structure is universal with respect to properties (1) and (2), i.e., if $\Spl'(X)$ is another such structure, then the identity map between their points is induced by a (necessarily unique) morphism $\Spl(X)\to \Spl'(X)$.
\end{enumerate}
One can prove that such a structure always exists, and by (3), it is unique up to canonical isomorphism.

By flatness, the rank and Chern classes of fibers $\cF_b$ in a family $\cF$ are locally constant on $B$. Hence, we get a decomposition of $\Spl(X)$ as a disjoint union of open subspaces $\Spl(r;c_1,c_2)$
 with fixed rank $r$ and Chern classes $c_1$ and $c_2$.  For any polarization $H$ of $X$, every $H$-stable sheaf is also simple. Since $H$-stability is open in families, this defines an open subspace $M^H(r;c_1,c_2)$ of $\Spl(r;c_1,c_2)$ which is a quasiprojective scheme, while in general $\Spl(r;c_1,c_2)$ may be not separated.

A family of simple sheaves $\cU$ parametrized by $\Spl(r;c_1,c_2)$ is called {\em universal}, if for any other family $\cF$ parametrized by a reduced scheme $B$, we have that $\cF$ is isomorphic to the pullback of $\cU$ via the induced morphism $B\to X$ (up to tensoring with the pullback of a line bundle on $B$).

In some cases, for instance $\Spl(1;0,0)=\Pic(X)$, a universal family exists. When it exists, it is unique up to tensoring with the pullback of a line bundle on $\Spl(r;c_1,c_2)$. Such a family may not exist; however, it exists \'etale locally on $\Spl(r;c_1,c_2)$. The key consequence for our work is the following: to prove that a subset $U$ of $\Spl(r;c_1,c_2)$ is open, it is enough to show that for any family of simple sheaves $\cF$ with the given invariants parametrized by $B$, the set $B_U:=\{b\in B\,|\, [\cF_b]\in U\}$ is open in $B$.

For the sake of completeness we recall what is the tangent space to $\Spl(r;c_1,c_2)$.

\begin{proposition}\label{tgt-obs} Let $F\in \Spl(r;c_1,c_2)$. Then $T_{[F]}\Spl(r;c_1,c_2)\cong Ext^1(F,F)$. In particular, if $F$ is a vector bundle then $T_{[F]}\Spl(r;c_1,c_2)\cong H^1(F^*\otimes F)$.
\end{proposition}
\begin{proof} See, for instance, \cite[Theorem 0.1]{Mu}.
\end{proof}

\begin{corollary}\label{spl_dim} Let $E\in \Spl(r;c_1,c_2)$ be a vector bundle. Then by Riemann-Roch \[
\dim T_{[E]}\Spl(r;c_1,c_2)=2-\chi(E^*\otimes E)=
c_1^2-2r\chi(E)+2r^2+2,\] 
where (again by Riemann Roch) $\chi(E)=2r+c_1^2/2-c_2$. In other words, \[
\dim T_{[E]}\Spl(r;c_1,c_2)=-2r^2+(1-r)c_1^2+2rc_2+2.
\]
\end{corollary}

\subsection{Isotropic and Lagrangian subspaces of symplectic algebraic spaces}

We recall that a symplectic form on a smooth algebraic space $M$ is a closed non-degenerate differential 2-form $\omega $. Therefore, for any $p\in M$, we have a skew symmetric perfect  pairing  $$\omega:T_pM\wedge T_pM \to \CC.$$ A symplectic algebraic space is a pair $(M,\omega)$ where $M$ is a smooth algebraic space and $\omega$ is a chosen symplectic form on $M$.

\vskip 2mm
 In \cite[Theorem 0.1]{Mu}, Mukai proved

\begin{theorem} The moduli space $\Spl(r;c_1,c_2)$ is  smooth and it has a natural symplectic structure, induced on the tangent space at the point $[E]$ by the Yoneda pairing on  $\ext^1(E,E)$.
\end{theorem}

We now recall the definition of isotropic and Lagrangian subspaces for a symplectic algebraic space. The main result of this paper will be a way to construct isotropic and Lagrangian subspaces in $\Spl(r;c_1,c_2)$.

\vskip 2mm
{\em Isotropic subspaces} are smooth irreducible locally closed subspaces $L$ of a symplectic algebraic space $(M,\omega)$ such that $\omega _{|L}=0$ (i.e. for any $q\in L$ the map $$\omega _{|L}: T_qL\wedge T_qL \hookrightarrow T_qM\wedge T_qM \to \CC$$ is zero). For any isotropic subspace $L$ of $(M,\omega)$ and for any $q\in L$, we have a pairing $T_qL\otimes (T_qM/T_qL) \to \CC$ inducing an inclusion $T_qL\hookrightarrow (T_qM/T_qL)^*$, and globally on $L$ an inclusion of vector bundles $T_L\hookrightarrow N_{L/M}^*$.

{\em Lagrangian subspaces} are isotropic subspaces $L$ of a symplectic algebraic space $(M,\omega)$ of maximal dimension $\dim L=\dim M/2$. In this case, the pairing $T_qL\otimes T_qM/T_qL \to \CC$ is a perfect pairing and $T_qL\cong (T_qM/T_qL)^*$ for all $q\in L$. Globally, $\omega$ induces an isomorphism $T_L\cong N_{L/M}^*$.

\vskip 2mm
\subsection{Generalized syzygy bundles.}

We will now introduce the first construction  of simple sheaves which will provide us isotropic and Lagrangian subspaces of the moduli space $\Spl(r;c_1,c_2)$.

\begin{definition}\label{def1}
Let $F$ be a globally generated vector bundle of rank $r$ on a K3 surface $X$. For any general subspace $W\subset  H^0(X, F)$ with $w:=\dim W\ge 2+r$ we define the {\em generalized syzygy bundle} $S$ on $X$ as the kernel of the evaluation map $$eval_{W}: W\otimes \mathcal O_X \to {\mathcal F}.$$ Therefore, $S$  fits inside a short exact sequence:
\begin{equation}
    \label{exact1}
e: \quad 0\to  S\to W\otimes \mathcal O_X \to  F\to 0, \text{ and }
\end{equation}
$$\begin{array}{rcl} rank(S)& = & \dim W- rank(F)= w-r, \\
c_1(S) & = & -c_1(F) \\
c_2(S)& = & -c_2(F)+c_1(F)^2.
\end{array}
$$
We will call  $S$ the {\em generalized syzygy bundle} associated to the couple $(F,W)$.
\end{definition}
Notice that when $F=\cO _X(L)$ is a very ample line bundle on $X$ and $W=H^0(X,L)$ we recover the classical definition of syzygy bundle, namely, the {\em syzygy bundle} $M_L$ associated to $\cO _X(L)$ is by definition the kernel of the evaluation map  $eval_{L}: H^0{\mathcal O}_X(L)\otimes \mathcal O_X \to {\mathcal O}_X(L)$. Thus, $M_L$ sits in the exact sequence:
\[
0\to M_L\to H^0({\mathcal O}_X(L)\otimes \mathcal O_X \to {\mathcal O}_X(L)\to 0.
\]
Syzygy bundles  on smooth curves of genus $g\ge 1$ are well understood thanks to work of  Butler \cite{Bu}, Beauville \cite{Be} and Ein-Lazarsfeld \cite{EL} among others. 
Syzygy bundles and generalized syzygy bundles on higher dimensional varieties arise in a variety of geometric and algebraic problems; and there has been great interest in trying to establish their stability with respect to a suitable ample line bundle as well as in describing their moduli spaces (see, for instance, \cite{BP}, \cite{B}, \cite{CL}, \cite{CMM}, \cite{ELM}, \cite{MM}, \cite{MS}, \cite{MS2} and \cite{MR}).
For recent results on generalized syzygy bundles and their moduli spaces on smooth projective varieties of dimension $\ge 2$ the reader can look at \cite{FM}.

\vskip 2mm
\noindent  {\bf Toy example.} Let $X\subset \PP^3$ a general quartic surface, $L=\cO_X(1)$ and $W\subset H^0(\cO_X(L))$ a general subspace of dimension 3 (Note that we are dealing with non-complete linear systems). We consider the syzygy bundle $M$ associated to  $(\cO_X(L), W)$. So, we have a short exact sequence:
$$
0 \to M \to W\otimes \cO_X \to \cO_X(L) \to 0.
$$
By construction $M$ is a rank 2  vector bundle on $X$ with Chern classes $(c_1(M),c_2(M))=(-L,4)$. $M$ is parameterized by an open subset of $Gr(3,H^0(\cO_X(L))$ which is a smooth rational quasi-projective variety of dimension 3. On the other hand, $M$ is simple. Hence $M$ lies in the moduli space $\Spl(2;-L,4)$ of rank 2 simple  sheaves on $X$ with Chern classes $(-L,4)$ which is a smooth symplectic algebraic space of dimension $\dim Ext^1(M,M)=6$. Indeed, $\dim Ext^1(M,M)=-\chi(M\otimes M^*)+2.$ If we tensor the exact sequence 
$$
0 \to \cO_X(-L) \to W^*\otimes \cO_X \to M^*\to 0
$$
with $M$ we get
$$
0 \to M(-L) \to W^*\otimes M \to M \otimes M^* \to 0.
$$
Therefore, we have $\dim Ext^1(M,M)=-3\chi (M)+\chi M(-L)+2=6$. 

\noindent {\em  Claim:} $\alpha :Gr(3,H^0(\cO_X(L)) \to \Spl(2;-L,4)$ is a Lagrangian locally closed embedding.

\noindent {\em Proof of the Claim.} Clearly $\alpha $ is injective and the differential 
$$\begin{array}{ccc} d \alpha :T_{[M]}Gr(3,H^0(\cO_X(L)) & \to &  T_{[M]}\Spl(2;-L,4)
\\ \parallel & & \parallel
\\
W^*\otimes H^1(M) & &  H^1(M\otimes M^*) \end{array} $$ is also injective. To prove that $\alpha $ is a Lagrangian locally closed embedding we only need to check that $W^*\otimes H^1(M)\subset H^1(M\otimes M^*)$  is an isotropic vector subspace. We observe that $W^*\otimes H^1(M)\cong H^0(M^*)\otimes H^1(M)$ and the map 
$$d \alpha :W^*\otimes H^1(M)\to H^1(M\otimes M^*)$$ is the cup product. To prove that $\wedge $ on $ H^0(M^*)\otimes H^1(M)$ is trivial it suffices to check it on simple tensors.   We take $v_1,v_2\in H^0(M^*)$ and $s_1,s_2\in H^1(M)$. So, $(v_1\otimes s_1)\otimes (v_2\otimes s_2)\in H^1(M\otimes M^*)\otimes H^1(M\otimes M^*)$ and when we compose to $H^2(M\otimes M^*)$ we get $v_1\otimes s_1(v_2)\otimes s_2\in H^0(M^*)\otimes H^1(\cO _X)\otimes H^1(M)=0$  because being $X$ a quartic surface in $\PP^3$ we have $H^1(\cO _X)=0$.

\vskip 2mm
\subsection {Extension of bundles.}
 Let us introduce the other construction of simple sheaves which will provide us new families of
isotropic and Lagrangian subspaces of the moduli space $\Spl(r; c_1, c_2)$.

\begin{definition}\label{DefExt}
Let $F$ be a  vector bundle of rank $r$ on a K3 surface $X$. For any subspace $V^*\subset  H^1(X,F^*)\cong Ext^1(F,\cO_X)$ of dimension $v:=\dim V\ge 1$ we define the {\em extension  bundle} $E$ on $X$ associated to the couple $(F,V)$
as the vector bundle that comes up from the induced extension of $F$ by $V\otimes \cO_X$. Therefore, $ E$  fits into a short exact sequence:
\begin{equation}
    \label{extension1}
e: \quad 0\to V\otimes \cO_X\to E \to F\to 0
\end{equation}
where $e\in H^1(V\otimes F^*)=Hom(V^*,H^1(F^*))$ is the inclusion $V^*\hookrightarrow H^1(F^*)$. Moreover, we have:
$$\begin{array}{rcl} rank(E)& = & rank(F)+ \dim V= r+v  , \\
c_1(E) & = & c_1(F) \\
c_2(E)& = & c_2(F).
\end{array}
$$
\end{definition}
Extension bundles
have been much investigated and used in many contexts. See, for instance, \cite{H}, \cite{CF} and \cite{L}.

\section{Moduli of syzygy bundles and extension bundles}

The following is Lemma 3.1 in \cite{FM} and it will be used many times.
\begin{lemma} \label{aux} Let $F $ be a simple vector bundle on a K3 surface $X$, If $F $ is generated by global sections, then $H^2(F)=0$ unless $F=\cO_X$.
\end{lemma}

\begin{Notation} We will denote by $$U:=U(r;c_1,c_2)\subset  \Spl(r; c_1,c_2)$$ the open locus parametrizing  globally generated rank $r$ simple vector bundles $F$ on $X$ such that $H^1(F)=0$.
\end{Notation}

\begin{remark}
By Lemma \ref{aux}, for any $F\in U$, we have $H^2(F)=0$. Therefore, $\dim H^0(F)=\chi(F)$ is constant on $U$. The fact that $U$ is open follows from the cohomology and base change theorem. 
\end{remark}

\begin{remark} (1) The results are only interesting when $U$ is non-empty. This is not very restrictive, since the theorems A and B of Serre guarantee that for any coherent sheaf $F$ and any ample line bundle $L$ on $X$, there always exists an integer $n_0$ such that for all $n\ge n_0$,   $F (nL)$ is generated by global sections
    and $H^1(F(nL))=0.$ Moreover, there is an isomorphism $$\Spl(r;c_1,c_2)\cong \Spl(r; c_1+rnL, c_2+(r-1)nc_1L+\binom{r}{2}n^2L^2)$$ which maps $F$ to $F(nL):=F\otimes \cO_X(nL)$. 

\vskip 2mm \noindent (2) For ample line bundles $\cO_X (L)$ on a K3 surface $X$ the hypothesis $H^1(\cO _X(L))=0
$ is always satisfied. Indeed, by Kodaira vanishing theorem we have  $H^1(\cO _X(L))=H^1(\cO _X(-L))=0
$.
\end{remark}

\begin{Notation} We will denote by $$U_e:=U_e(r;c_1,c_2)\subset  \Spl(r; c_1,c_2)$$ the open locus parametrizing rank $r$ simple vector bundles $F$ on $X$ such that $H^0(F)=H^0(F^*)=0.$
We note that for such an $F$, we have  $h^1(F)=-\chi (F)$.
\end{Notation}

\begin{remark} Again, our results are only  interesting if $U_e$ is nonempty. There are interesting examples of this. For instance, for any choice of ample line bundle $H$ on $X$, we have $M^H(r;0,c_2)\subset U_e$. 
For a line bundle example, assume that $X$ contains two disjoint smooth rational curves $C_1$ and $C_2$. Then any line bundle $L_n:=\cO_X(nC_1-nC_2)$ with $n\ge 1$ will be in $U_e$, with $h^0(L_n)=h^2(L_n)=0$, and thus by Riemann-Roch \[h^1(L_n)=-\chi (\cO_X(nC_1-nC_2))=-\frac{(nC_1-nC_2)^2}{2}-2= 2n^2-2\] 
In particular, we need $n>1$ to ensure that a nonsplitting extension exists.
\end{remark}

\begin{definition}
We define
$$\cG _U:=\{(F,W) \mid F\in U, W\subset H^0(F), \dim (W)=w\}.$$
The natural projection $\pi: \cG_U\to U$ is a Grassmann bundle and  $\cG_U$ is a smooth algebraic space of dimension $$\dim \cG_U=\dim \Spl(r;c_1,c_2)+ w(v-w)$$ where $v:=\dim H^0(F)$. For more details, see \cite[Definition 3.4]{FM}.
Moreover, we define 
$$\cG^0 _U:=\{(F,W)\in \cG _U  \mid eval_W \text{ is surjective}\}$$
which is open in $\cG _U$ and we keep denoting  by $\pi $ the natural projection $\pi : \cG^0_U \to U$.
\end{definition}

\begin{definition}
As in \cite[Definition 3.4]{FM}, we define 
 $$\cP _{U_e}:=\{ (F,V) \mid F \in U_e(r;c_1,c_2) \text{ and } V^*\subset H^1(F^*)\}.$$
More precisely, we consider the open substack $\cU_e$ in the moduli stack of simple sheaves $\SSpl(r;c_1,c_2)$ with the same points as $U_e$ (for details, see \cite[Section 2.1]{FM}  and, in particular, Proposition 2.5 and Definition 3.4). We denote by $\cF_{\cU_e}$ the universal bundle on $X\times {\cU_e}$, and let $\cP_{\cU_e}$ be the relative Grassmann bundle of rank $v$ subspaces in the 
vector bundle $R^1\bar p_*\cF_{\cU_e}$, where $\bar p:X\times {\cU_e}\to {\cU_e}$ is the projection. Let $\bar \pi: \cP_{\cU_e}\to  \cU_e$ be the projection  map and  denote by $\cF$ the  pullback vector bundle $\cF:=(\id_X\times \bar \pi)^*\cF_{\cU_e}$ on $X\times \cP_{\cU_e}$. \\

We  define,  in analogy with \cite{FM}, the algebraic space $\cP _{U_e}$ to be the associated coarse moduli space.
 The natural projection $\pi: \cP_{U_e}\to U_e$ is a Grassmann bundle. Therefore,  $\cP_{U_e}$ is a smooth algebraic space, of dimension $$\dim \cP_{U_e}=\dim \Spl(r;c_1,c_2)+v(u-v)$$ where $$u=\dim H^1(F^*)=-\chi(F^*)=-2r-c_1^2/2+c_2.$$ Moreover, $\cP_{U_e}\to U_e$ is proper.
 \end{definition}

\subsection{Basic properties of generalized syzygy bundles and extension bundles}

In  this subsection, we collect the general results on generalized syzygy bundles as well as on extension bundles needed in the sequel. As a first result we will prove that  the generalized syzygy bundle $S$ and the extension bundle $E$ associated to $(F ,W)\in \cG _U^0$ and $(F,V)\in \cP_{U_e}$, respectively, are simple. We then state  the injectivity of the natural morphism $\cG_U^0\to \Spl(r';c_1',c_2')$ induced by the set  theoretic map $(F,W)\mapsto S$. In a similar vein, we prove that the morphism $\cP_{U_e}\to \Spl(r+v;c_1,c_2)$ inducing the set map $(F,V)\mapsto E$ is also injective.

\vskip 2mm

\begin{remark}
Let $F$ be a nontrivial vector bundle on a K3 surface $X$ and $w\in \NN_{>0}$. The following conditions are equivalent:\begin{itemize}
\item[(1)] For a generic $W\subset H^0(F)$, the evaluation map $W\otimes \cO_X \to F$ is surjective; and
\item[(2)] $F$ is generated by global sections and $w\ge rank(F) +2$.
\end{itemize}
\end{remark}

\begin{lemma} \label{simple} Let $X$ be a  K3 surface 
 and let $S$ be a generalized syzygy bundle associated to $(F,W)\in \cG_U^0$. Then:
\begin{itemize}
    \item[(a)] $H^0(S)=0$,
    \item[(b)] There is a natural isomorphism $W^*\to H^0(S^*)$,
    \item[(c)] $S$ is simple.
    \end{itemize} 
\end{lemma}

\begin{proof} See \cite[Proposition 3.5]{FM}.
\end{proof}

\begin{remark}
Without the hypothesis $H^1(F)=0$, Lemma \ref{simple} is false. Take $X\subset \PP^3$ a general quartic surface. Let $F$ be a rank 2 vector bundle on $X$ defined as the cokernel of a map $\cO_{X}(-3)^2 \to \cO_{X}^4 $ given by a general $2\times 4$ matrix with entries forms of degree 3. $F$  is a stable (and, hence, simple) rank 2 vector bundle on $X$, generated by global sections, $H^1(F) =H^1(F^*)\ne 0$ and the syzygy bundle $S $ associated to $(F , W=H^0(F))$ is not simple.
\end{remark}

\begin{definition}\label{defalpha} (\cite[Definition 3.7]{FM}) We define 
$$  \alpha: \cG_U^0   \to   \Spl(r';c_1',c_2') 
$$ 
as the natural morphism which  extends the set theoretic map $$(F ,W)  \mapsto  S:= ker(  W\otimes \cO_X \twoheadrightarrow F).$$
\end{definition}

The map $\alpha $ was deeply studied in \cite{FM} where we prove the following.
\begin{proposition}\label{inj}
    \begin{itemize}
    \item[(1)] The morphism $\alpha :\cG_U^0   \to   \Spl(r';c_1',c_2')$ is injective  \cite[Proposition 4.1]{FM}.
    \item[(2)] The differential map $d \alpha : T_{[(F,W)]}\cG_U^0 \to  T_{[S]}\Spl(r'; c_1',c_2')$
is injective  \cite[Proposition 4.2]{FM}.
    \item[(3)] The morphism $\alpha :\cG_U^0   \to   \Spl(r';c_1',c_2')$ is a locally closed embedding  \cite[Theorem 1.1]{FM}.
    \end{itemize}
\end{proposition}

In the remaining part of this subsection we deal with extension bundles and we get completely analogous results to those obtained by generalized syzygy bundles.

\begin{lemma} \label{simple2} Let $X$ be a K3 surface. The extension bundle $E$ associated to $(F, V)\in \cP_{U_e}$ is simple. Moreover, $H^0(E^*\otimes E)\to H^0(E^*\otimes F)$ is an isomorphism.
\end{lemma}
\begin{proof}
We dualize the exact sequence  (\ref{extension1}) and we get
\begin{equation}
    \label{extension2}
0\to F^*\to E^* \to V^*\otimes \cO_X\to 0.
\end{equation}
Since $V^*\otimes H^0(\cO_X)\hookrightarrow H^1(F^*)$ and $H^0(F^*)=0$, we deduce that $H^0(E^*)=0$. Tensoring the short exact sequence (\ref{extension1}) with $E^*$ and taking cohomology we obtain $h^0(E^*\otimes E)\le h^0(E^*\otimes F).$ 
Finally,  the exact sequence (\ref{extension2}) tensored with $F$ yields 
\begin{equation}
    \label{extension3}
0\to F^*\otimes F \to E^*\otimes F \to V^*\otimes F\to 0
\end{equation}
and  taking once more cohomology we get, using the hypothesis $H^0(F)=0$, that $h^0(E^*\otimes F)=1$ which allows us to conclude that  $H^0(E \otimes E^*)=\CC$. Therefore, $E$ is simple.
\end{proof}

\begin{remark}
 Without the assumption $H^0(F)=H^0(F^*)=0$ the last Lemma is false. Take $X\subset \PP^3$ a general quartic surface. Let $F^*$ be a rank 2 simple  vector bundle on $X$ defined as the cokernel of a general map $\cO_{X}(-1)^2 \to \cO_{X}^4 $. So, we have a short exact sequence 
$0 \to F \to \cO_X^4\to \cO_X(1)^2\to 0$ which gives us $H^0(F^*)=\CC^4$ and $H^1(F)=H^1(F^*)\ne 0$. A non-zero element $e\in Ext^1(F, \cO_X)=H^1(F^*)$ gives an extension bundle $E$ of rank 3 sitting in an exact sequence:
\begin{equation}\label{auxtoy2}
0\to \cO_X \to E \to F \to 0.
\end{equation}
Tensoring with $E^*$ we get $h^0(E^*)\le h^0(E \otimes E^*)$. But dualizing (\ref{auxtoy2}) and taking cohomology we get $\CC^4=H^0(F^*)\hookrightarrow H^0(E^*)$ which allows us to conclude that $E$ is not simple.
\end{remark}

\begin{definition}\label{def_beta} We define  
$$\begin{array}{rcl}  \beta: \cP_{U_e} & \to & \Spl(\overline{r};c_1,c_2) \\
(F ,V) & \mapsto & E
\end{array}
$$ 
where $E$ is the extension bundle associated to $(F,V)$ and, in analogy with the definition of $\alpha$ in \cite[Definition 3.7]{FM}, we extend this pointwise definition to a morphism of algebraic spaces as follows, 

As usual we call $\bar p:X\times \cU_e \to \cU_e$ and $\tilde p: \cP_{\cU_e}\times X\to \cP_{\cU_e}$ the natural projections. On $X\times \cP_{\cU_e}$ we have a universal sequence \[
0\to \tilde p^*\cV\to \cE\to \cF\to 0,
\]
where $\bar \pi: \cP_{\cU_e}\to  \cU_e$ is the projection  map, $\cF:=(
\id_X\times \bar \pi)^*\cF_ {\cU_e}$ is the pullback of the universal bundle $\cF_{\cU_e}$ on $X\times \cU_e$, and $\cV^*\subset R^1\tilde p_*\cF$ is the universal subbundle. The family $\cE$ of simple vector bundles induces a morphism  $$\bar \beta :\cP_{\cU_e} \to \SSpl(\bar r;c_1,c_2)$$ which extends the set theoretic map  $(F,V)\mapsto E$. The morphism of stacks 
 $\bar \beta$ induces a morphism $\beta$  of algebraic spaces from $\cP_{U_e}$ to $\Spl(\bar r;c_1,c_2)$. The commutative diagram
\[
\xymatrix{\cP_{\cU _e} \ar[r]^{\bar \beta \,\, \,\ }\ar[d] &\SSpl(\bar r;c_1,c_2) \ar[d]\\
\cP_{U_e} \ar[r]^{\beta \,\,\,\ } &\Spl(\bar r;c_1,c_2)}
\]

\noindent is cartesian and the vertical maps are $\Gm$-gerbes. Therefore, the following properties of $\beta $ and $\overline{\beta }$ imply each other; is injective, has injective differential, is a locally closed embedding, is an open embedding.

\end{definition}

\begin{proposition}\label{inj2}    Let $X$ be a  K3 surface. The morphism
$$\beta: \cP_{U_e}  \to  \Spl(\overline{r};c_1,c_2) 
$$ 
  is injective.
\end{proposition}
\begin{proof}
Since by hypothesis $H^0(F)=0$, we first recover $V$ as $H^0(E)$ using the exact cohomology sequence associated to  the exact sequence (\ref{extension1}). We recover $F$ as the cokernel of the evaluation map $V\otimes \cO_X\to E$. We note in passing that $h^0(F\otimes E^*)=1$, so the morphism $E\to F$ is unique up to scalar.
\end{proof}

In next section we will prove that $\beta $  has injective differential (see Proposition \ref{difBetainj}) and it is a locally closed embedding (see Lemma \ref{betaemb}).

\subsection{The moduli spaces $\Syz^w_\cL$ and $\cP_\cL$}

\begin{Notation}
    Let $\cL\subset \Spl(r;c_1,c_2)$ be {\em any} irreducible, smooth, locally closed subspace. We define $\cL_U:=\cL\cap U$ and $\cG^0_\cL$ to be the intersection of $\pi^{-1}(\cL_U)\subset \cG_U$ with $\cG_U^0$.  The algebraic space $\cG^0_\cL$ is irreducible, smooth and there is a natural morphism from $\cG^0_\cL$ to $\Spl(r';c_1',c_2')$ which is the restriction of $\alpha$ from Definition \ref{defalpha}; it is injective by Proposition \ref{inj}. The image  of $\cG^0_\cL$ via the locally
closed embedding $\alpha :\cG^0_U \to  \Spl(r';c_1',c_2')$ will be denoted by $\Syz_{\cL}^w$ and called the {\em  moduli of $w$-syzygies of bundles.}
\end{Notation}

In the rest of the paper we will always assume that $\cG^0_\cL$ is non-empty.

\begin{proposition}\label{dims}  Let $X$ be a  K3 surface  and  let $\cL\subset \Spl(r;c_1,c_2)$ be a locally closed subspace of dimension $\frac{1}{2}\dim \Spl(r;c_1,c_2)$.
Then,  it holds:
$$ \dim \cG^0_\cL =\frac{1}{2} \dim \Spl(r';c_1',c_2').$$
\end{proposition}

\begin{proof}By Proposition \ref{tgt-obs} we have $\dim \Spl(r;c_1,c_2)=\dim Ext^1(F,F)$ and $\dim \Spl(r';c_1',c_2')=\dim Ext^1(S,S)$.  Let us check that $$\dim \cG^0_\cL=\frac{1}{2}\dim \Spl(r';c_1',c_2').$$ By construction $S$ moves in a Grassmann bundle over $\cL  $. By Lemma \ref{aux} we have $h^0(F)= \chi (F)$ and hence $S$ moves in a moduli space $\cG^0_{\cL}$ of dimension: $$\begin{array}{rcl}\dim  \cG^0_{\cL} & = &\dim  \cL  + \dim Gr(w, H^0(F)) \\
& = &\dim \cL +w(h^0(F)-w) \\ & = & \frac{1}{2}h^1(F\otimes F^*)+w(\chi (F)-w).\end{array}$$ Let us compute $h^1(S\otimes S^*)$.  Since $S$ is simple, we have $\chi (S\otimes S^*)=2-h^1(S\otimes S^*)$. Using the exact sequence 
\begin{equation}\label{dualtensorS}
0\to F^*\otimes S\to W^*\otimes S\to S^*\otimes S\to 0
\end{equation}
we get:
$$
\chi (S\otimes S^*)=w\chi(S)-\chi(F^*\otimes S).
$$
By (\ref{exact1}) we have $\chi(S)=2w-\chi(F)=2w-h^0(F)$ and by 
$$
0\to F^*\otimes S\to W\otimes F^*\to F\otimes F^*\to 0
$$
we have $\chi(F^*\otimes S)=w\chi(F)-2+h^1(F\otimes F^*)$. Putting altogether we obtain:
$$\begin{array}{rcl}
\chi (S\otimes S^*) & = & w\chi(S)-\chi(F^*\otimes S) \\
& = & w(2w-h^0(F))-(w\chi(F)-2+h^1(F\otimes F^*) )\\
& = & 2-  \dim Gr(w, H^0(F))-h^1(F \otimes F^*)
\end{array}
$$
and, hence, $h^1(S\otimes S^*)=2 \dim Gr(w,\chi(F))+h^1(F\otimes F^*)$ which proves what we want.
\end{proof}

\begin{remark}
    The proposition suggests that, if $\cL\subset \Spl(r;c_1,c_2)$ is a  locally closed Lagrangian subspace, then $\Syz_{\cL}^w\subset \Spl(r';c_1',c_2')$ is also a Lagrangian. Indeed, it is a key step in the proof of our main theorem about syzygy bundles, as it allows to go directly from the statement about locally closed isotropic subspaces to that about locally closed Lagrangian subspaces.
\end{remark}

Let $\cL\subset \Spl(r;c_1,c_2)$ be {\em any} irreducible, smooth, locally closed subspace. We define $\cL_{U_e}:=\cL\cap U_e$ and 
let $\cP_\cL$ to be  $\pi^{-1}(L_{U_e})\subset \cP_{U_e}$.
The algebraic space $\cP_\cL$ is irreducible, smooth and there is a natural morphism from $\cP_\cL$ to $\Spl(\overline{r};c_1,c_2)$ (the restriction of $\beta $ from Definition \ref{def_beta}) which is injective by Proposition \ref{inj2}.
The image of $\cP_\cL$ via the locally
closed embedding $\beta :\cP_{U_e} \to  \Spl(\bar r;c_1,c_2)$ will be denoted by $\ext ^v_{\cL}$
and called the  {\em  moduli of $v$-extension bundles} in $\cL_{U_e}$.

\vskip 2mm
In the rest of the paper we will always assume that $\cP_\cL$ is non-empty.

\vskip 2mm
Analogously to the case of generalized syzygy bundles, next result  suggests that if $\cL\subset \Spl(r;c_1,c_2)$ is a  locally closed Lagrangian subspace, then $\ext ^v_{\cL}\subset \Spl(\overline{r};c_1,c_2)$ is also a Lagrangian.

\begin{proposition} \label{dim2} Let $X$ be a K3 surface  and let $\cL \subset \Spl(r;c_1,c_2)$ be a locally closed subspace of dimension $\frac{1}{2}\dim \Spl(r;c_1,c_2)$. Then, it holds:
$$\dim \cP_\cL=\frac{1}{2}\dim \Spl(\overline{r};c_1,c_2).$$
\end{proposition}

\begin{proof}  It is analogous to the proof of Proposition \ref{dims} and we omit the details.
 By construction we have:
 $$\begin{array}{rcl}\dim  \cP_\cL  & = &
  \dim \cL  + \dim Gr(v, H^1(F))\\
  & = &\dim \cL  +v(h^1(F)-v) \\ & = & \frac{1}{2}h^1(F\otimes F^*)-v(\chi(F)+v).\end{array}$$ 
 Let us compute $h^1(E\otimes E^*)$.  Since $E$ is simple, we have $\chi (E\otimes E^*)=2-h^1(E\otimes E^*)$. On the other hand, we have:
$$\begin{array}{rcl}
\chi (E\otimes E^*) & = & v\chi(E^*)+\chi(F\otimes E^*) \\
& = & v(2v+\chi(F))+ \chi(F\otimes E^*) \\
& = & v(2v+\chi(F))+ v\chi(F)+ \chi(F\otimes F^*) \\
& = & 2v(v+\chi(F))+  \chi(F\otimes F^*)
\end{array}
$$
and, hence, $h^1(E\otimes E^*)=h^1(F\otimes F^*)-2v(v+\chi(F))$ which finishes our proof.
\end{proof}


\section{Proof of Theorem \ref{mainthm1}}

We keep the notation of the previous section 
and we start recalling from \cite[Proposition 2.16]{FM} a geometric interpretation of the tangent space to $\cG^0_U$  as well as the relative tangent space to the morphism $\pi:\cG^0_U\to U$.

\begin{proposition} 
With the above notation we have canonical isomorphisms:
\begin{itemize}
 
\item[(1)] $T_{[(F,W)]}(\pi ^{-1}(F))  =  W^*\otimes H^1(S)$,  and
 \item[(2)] $ T_{[(F , W)]}\cG_U  \cong  H^0(S^*\otimes F)/W^*\otimes W $.
\end{itemize}
\end{proposition}

We first  prove that the morphism $\alpha$ is isotropic when $F$ is rigid, hence  the isolated point $[F]$ in $\Spl(r;c_1,c_2)$ is by definition Lagrangian.

\begin{proposition}\label{isot1} $d\alpha: W^*\otimes H^1(S)\hookrightarrow H^1(S\otimes S^*)$ is an isotropic vector subspace.
\end{proposition}

\begin{proof} The cohomology sequence associated to the exact sequence 
$$
0\to F^* \to W^*\otimes \cO_X\to S^* \to 0
$$
gives us $W^*\cong H^0(S^*)$. Therefore, $W^*\otimes H^1(S)\cong H^0(S^*)\otimes H^1(S)$ and the map $d \alpha $ is  the cup product. To prove that $\wedge $ on $H^0(S^*)\otimes H^1(S)$ is trivial it is enough to check on a basis and hence on simple tensors.

We choose $v_1,v_2\in H^0(S^*)$ and $s_1,s_2\in H^1(S)$. Therefore, $(v_1\otimes s_1)\otimes (v_2\otimes s_2)\in H^1(S\otimes S^*)\otimes H^1(S\otimes S^*)$ and when we compose to $H^2(S\otimes S^*)$ we get $v_1\otimes s_1(v_2)\otimes s_2\in H^0(S^*)\otimes H^1(\cO _X)\otimes H^1(S)$ which is zero because being $X$ a K3 surface we have $H^1(\cO _X)=0$. 

Notice that we get zero in $H^2(S\otimes S^*)\cong H^2(End(S))$ which is expected because $S$ is simple and thus the trace map $$tr:H^2(End(S)) \to H^2(\cO_X)=\CC$$ is an isomorphism.
\end{proof}

Since $d\alpha $ is injective, there is an induced linear map 
$$  H^1(F\otimes F^*) \to H^1(S\otimes S^*) / W^*\otimes H^1(S) .
$$
Let us construct it in a natural way. To this end, we take the extension $0\ne e\in Ext^1(F,S)\cong H^1(F^*\otimes S)$ and we build the commutative diagram

$$\xymatrix{W^*\otimes H^0(F)\ar[r]_{\wedge (e) \ \ \ \ \ \ } \ar[d]^{\phi}  & W^*\otimes H^1(F\otimes F^*\otimes S) \ar[d] \ar[r]^{\ \ \ \ \ \ comp}  &    W^*\otimes H^1( S)\ar[d]^{\varphi} \\
H^0(S^*\otimes F)\ar[r]_{\wedge (e)\ \ \ \ \ \  }   &  H^1(S^*\otimes F\otimes F^*\otimes S) \ar[r]^{\ \ \ \ \ \ comp}  &    H^1( S^*\otimes S) 
}$$
which induces a map on cokernels
$$Coker(\phi)=H^1(F^*\otimes F) \longrightarrow Coker(\varphi)=H^1(S^*\otimes S)/W^*\otimes H^1(S).$$

Therefore, we have the commutative diagram:
$$\xymatrix{H^0(S^*\otimes F)\ar[rr] \ar[dd] \ar[rd] & &    H^1(F^*\otimes F)\\
&  H^0(S^*\otimes F)/W\otimes W^* \ar[ru] \ar@{^(->}[ld] & \\
H^1(S^*\otimes S)  &  & &}$$

\begin{proposition} \label{restrict2-forms}  With the above notation we have that the pullback of the closed non-degenerate holomorphic 2-forms on $H^1(F^*\otimes F)$ and $H^1(S^*\otimes S)$   to $T_{[(F , W)]}\cG_{U} \cong H^0(S^*\otimes F)/W^*\otimes W$ are the same.\end{proposition}

\begin{proof} It is enough to prove that the  pullback of the symplectic  2-forms on $H^1(F^*\otimes F)$ and $H^1(S^*\otimes S)$ to $H^0(S^*\otimes F)$ are the same.
This is a special case of the following lemma.             
\end{proof}

The following lemma holds on any projective complex manifold, except of course the forms $\sigma_E$ and $\sigma_F$ are  not in general symplectic, just two-forms valued in $H^2(\cO_X)$.
\begin{lemma}
    Let \[
    0\longrightarrow E\longrightarrow Q\xrightarrow{\,\,\,\pi\,\,\,} F\longrightarrow 0
    \]
    be an exact sequence of vector bundles on $X$, corresponding to a class $e\in H^1(F^*\otimes E)$.
    By pre and post composition, $e$ induces linear maps $\beta:H^0(E^*\otimes F)\to H^1(E^*\otimes E)$ and $\gamma:H^0(E^*\otimes F)\to H^1(F^*\otimes F)$. Let $\sigma_E$ be the skew symmetric two form on $H^1(E^*\otimes E)$ with values in $H^2(\mathcal O_X)$ induced by composition and trace, and $\sigma_F$ the analogous form on $H^1(F^*\otimes F)$. Then $\beta^*\sigma_E=-\gamma^*\sigma_F$.
\end{lemma}

\begin{proof}

We want to describe $e\in H^1(F^*\otimes E)$ as a cycle in Dolbeault cohomology. We choose a $C^\infty$ splitting $s:F\to Q$ and compute $\deb (s)\in A^{0,1}(F^*\otimes Q)$. We have that $\pi\deb(s)=\deb (\pi\circ s)=0$ because $\pi$ is holomorphic and  $\pi\circ f=\id_F$  is also holomorphic.

This implies that if we write in local coordinates $\deb(s)=\sum a_id\bar z_i$, then $a_i$ are not just local sections of $F^*\otimes Q$, but local sections of $F^*\otimes E$, since $\pi(a_i)=0$ (here we identify $E$ with its image inside $Q$).

We now choose $\phi, \psi\in H^0(E^*\otimes F)$. We get from $\phi$ and $e$ induced classes $\phi'\in H^1(F^*\otimes F)$ and $\phi''\in H^1(E^*\otimes E)$. We write explicitly the corresponding Dolbeault cocycles as $\phi'=\sum (\phi\circ a_i) d\bar z_i$, and $\phi^{\prime\prime}=\sum_i(a_i\circ \phi)d\bar z_i$; we do the same for $\psi$, $\psi'$ and $\psi''$.

To compute the image of $\phi'\wedge\psi'$ in $H^2(F^*\otimes F)$ we compose the local sections and wedge the differential forms. We obtain \[
\phi'\wedge\psi'=\sum_{i,j}(\phi\circ a_i\circ \psi\circ a_j)d\bar z_i\wedge d\bar z_j.
\]

Similarly we have that in $H^2(E^*\otimes E)$
\[
\phi''\wedge\psi''=\sum_{i,j}(a_i\circ \phi\circ a_j\circ \psi)d\bar z_i\wedge d\bar z_j.
\]

This implies that 
\[
-\phi''\wedge\psi''=\psi''\wedge\phi''=\sum_{i,j}(a_i\circ \psi\circ a_j\circ \phi)d\bar z_i\wedge d\bar z_j.
\]
In particular, we have 
\[
\tr(\phi'\wedge\psi')=\sum_{i,j}\tr(\phi\circ a_i\circ \psi\circ a_j)d\bar z_i\wedge d\bar z_j\in H^2(\cO_X)
\]
and
\[
-\tr(\phi''\wedge\psi'')=\sum_{i,j}\tr(a_i\circ \psi\circ a_j\circ \phi)d\bar z_i\wedge d\bar z_j.
\]
We conclude the proof by observing that, for each $i$ and $j$, \[
\tr(\phi\circ a_i\circ \psi\circ a_j)=\tr(a_i\circ \psi\circ a_j\circ \phi)
\] 
hence $\tr(\phi'\wedge\psi')=-\tr(\phi''\wedge\psi'')$ in $A^{0,2}(\cO_X)$, therefore they define the same class in $H^2(\cO_X)$.
    \end{proof}

We are now ready to prove the main result of this section. 

\begin{theorem}
\label{pfthm1}
A smooth, irreducible,  locally closed subspace $\cL\subset \Spl(r;c_1,c_2)$ is 
isotropic (resp. Lagrangian) if and only if $$\alpha _{| \cG_\cL^0}:\cG^0_\cL \longrightarrow \Spl(r';c_1',c_2')$$ is an isotropic (resp. Lagrangian) locally closed embedding.
\end{theorem}
\begin{proof} We consider the diagram
 $$
 \xymatrix{\cG^0_\cL 
  \ar[r] \ar[d]^{\pi}  &   \Spl(r';c_1',c_2')\\
\cL_U=\cL \cap U \subset U\subset \Spl(r;c_1,c_2) & & 
}$$
By Proposition \ref{restrict2-forms}  the pullback of the symplectic 2-form on $\Spl(r';c_1',c_2')$ is the same as the pullback of the 2-form on $\Spl(r;c_1,c_2)$. On the other hand, the differential of $\pi $ is surjective because it is a Grassmann bundle. Hence the pullback of the symplectic 2-form on $U$
 is zero if and only if its restriction to $\cL_U$ is zero. Therefore,
 $$\alpha _{| \cG^0_\cL}:\cG^0_\cL \longrightarrow \Spl(r';c_1',c_2')$$
is an isotropic  immersion if and only if $\cL\subset \Spl(r;c_1,c_2)$ is an isotropic subspace. 

 The statement about Lagrangians follows by the dimension computation established in Proposition \ref{dims}. 
 \end{proof}

 \begin{corollary}\label{FirstCoro}  Let $X$ be a K3 surface, $L$ an ample line bundle which is generated by global sections and $3\le w\le h^0(\cO_X(L))$. Then the moduli space $Syz^w_{[L]}$ of syzygy bundles is a non-empty, rational Lagrangian subspace of $\Spl(w-1;-c_1(L),L^2)$ which is also a smooth rational quasiprojective variety (open in $Gr(w,H^0(L))$).
 \end{corollary}
 \begin{proof}
 It follows from Theorem \ref{mainthm1} taking into account that $\cO_X(L)$ is simple and rigid, hence the point $[L]$ is equal to $\Spl(1;c_1(L),0)$ and thus also Lagrangian. The condition $H^1(L)=0$ follows from $L$ being ample. 
 \end{proof}


\section{Proof of Theorem \ref{mainthm2}}

We keep the notation of the previous sections and we consider the morphism
  $$\begin{array}{crcl} \beta : &\cP_{\cU_e} & \longrightarrow & \Spl(\overline{r};c_1,c_2) \\
 & (F,V) & \mapsto & E.
 \end{array}
 $$
 By Proposition \ref{inj2}, the map $\beta $ is injective. To show that it has injective differential and whether it is an isotropic (resp. Lagrangian) locally closed embedding we need to  describe the tangent  space to $\cP_{U_e}$  as well as the relative tangent space to the morphism $\pi:\cP_{U_e}\to U_e$. 

\begin{proposition} 
With the above notation we have canonical isomorphisms:
\begin{itemize}
\item[(1)] $
T_{[(F,V^*)]}(\pi ^{-1}(F))  \cong  V\otimes H^1(E^*)$,  \text{ and }
 \item[(2)]$ T_{[(F , V^*)]}\cP_{U_e}  \cong  H^1(E\otimes F^*)/V^*\otimes V$.
\end{itemize}
\end{proposition}
\begin{proof} 

(1) Being $\pi^{-1}(F)=Gr(v,H^1(X,F^*))$ we get
 $$\begin{array}{rcl}T_{[(F,V^*)]}(\pi ^{-1}(F)) & \cong & Hom(V^*,H^1(X,\cF^*)/V^*) \\
 & \cong & V\otimes H^1(X,F^*)/V^*\\
& \cong & V\otimes H^1(X,E^*) \\
 & \cong & H^1(X,E^*\otimes V).
\end{array}
$$

(2)  We will prove the statement for $ T_{[(F , V^*)]}\cP_{\cU_e} $ instead, since $\cP_{\cU_e}\to \cP_{U_e}$ is a $\Gm$-gerbe. Let $\widetilde\cP_{\cU_e}$ be the principal $GL(v)$ bundle of reference frames (bases) in $\cU_e$ over $\cP_{\cU_e}$.
A point in $\widetilde\cP_{\cU_e}$ is an exact sequence \[
0\to \cO_X^{\oplus v}\to E\to F\to 0
\]
or equivalently its dual
 \[
0\to F^* \to E^*\to \cO_X^{\oplus v}\to 0.
\]

Its deformations are induced by deforming the bundle $E^*$ and the map $E^*\to \cO_X^{\oplus v}$ while keeping the target $\cO_X^{\oplus v}$ fixed.

If $f:A\to B$ is a morphism of vector bundles on a scheme $Y$, the tangent space to deforming the pair $(A,f)$ while keeping $B$ fixed is $\bH^1(Y, T_f)$, where $T_f:=[A^*\otimes A\to A^*\otimes B]\in D^{-1,0}(Y).$

Applying this to our case, we get that the tangent space is $H^1(E\otimes F^*)$. Quotienting by $V\otimes V^*$ (the relative tangent space of $\widetilde \cP_{\cU_e}\to \cP_{\cU_e}$ at the point $(F,V^*)$) gives us the tangent space to $\cP_{\cU_e}$.
\end{proof}

\vskip 2mm

\begin{proposition}\label{difBetainj} The differential map
$$d\beta : T_{[(F , V^*)]}\cP_{U_e}  \cong  H^1(E\otimes F^*)/V^*\otimes V \to T_{[E]}(Slp(r';c_1',c_2'))  \cong  H^1(E \otimes E^*)
$$
is injective.
\end{proposition}
\begin{proof} Since $V\cong H^0(E)$ the result immediately follows from the exact cohomology sequence associated to the exact sequence:
$$
0 \to F^*\otimes E \to E^*\otimes E \to V^*\otimes E \to 0.
$$
\end{proof}

\begin{lemma} \label{betaemb} The morphism $\beta:\cP_{U_e}\to\Spl(\overline{r};c_1,c_2)$ is a locally closed embedding.
\end{lemma}
\begin{proof} 
In analogy with \cite{FM} Section 4, we will prove (equivalently) that the morphism $\bar\beta:\cP_{\cU_e}\to \SSpl(\overline{r};c_1,c_2)$ is a locally closed embedding by
first constructing a locally closed substack $\EExt^v_{\cU_e}$ of $\SSpl(\overline{r};c_1,c_2)$ and then proving that the morphism $\bar\beta:\cP_{\cU_e}\to \SSpl(\overline{r};c_1,c_2)$ induces an isomorphism $\cP_{\cU_e}\to \EExt^v_{\cU_e}$.

We define an open substack $\cZ_0$ in $\SSpl(\overline{r};c_1,c_2)$ by requiring its points $[E]$ are bundles such that 
    $H^0(E^*)=H^2(E)=0$, which is open by cohomology and base change. We define a locally closed substack $\cZ_1$ of $\cZ_0$, whose points are the vector bundles $E$ satisfying $h^0(E)=v$, by the $v$-th Fitting ideal of the sheaf $R^2q_*(\cE^*\otimes \omega_X)$ where $q:X\times \cZ_0\to \cZ_0$ is the projection and $\cE$ is the universal bundle on $X\times \cZ_0$, in analogy with \cite[Section 4.2]{FM}; by the properties of the Fitting ideal, the morphism $\bar\beta$ factors via $\cZ_1$.

     We define an open substack $\cZ_2$ of $\cZ_1$ whose points are the vector bundles $E$ such that the evaluation map $H^0(E)\otimes\cO_X\to E$ is injective and has as cokernel a simple vector bundle. The cokernel of the evaluation morphism $q^*q_*\cE \to \cE$ on $X\times \cZ_2$ is a family of simple vector bundles with rank $r$ and Chern classes $c_1$, $c_2$. Therefore, it defines a morphism $\theta:\cZ_2\to \SSpl(r;c_1,c_2)$.
     
     We define $\EExt^v_{\cU_e}$ to be $\theta^{-1}(\cU_e)$, which is open in $\cZ_2$ as $\cU_e$ is open in $\SSpl(r;c_1,c_2)$. By construction, $\bar\beta$ factors via $\EExt^v_{\cU_e}$.
     We have an exact sequence on $X\times \EExt^v_{\cU_e}$
     \[
     0\to q^*q_*\cE \to \cE \to (\id_X\times\theta)^*\cF_{\cU_e}\to 0
     \]
     which realizes $(q_*\cE)^*$ as a rank $v$ subbundle of $\theta^*(R^1\bar p_*\cF_{\cU_e})$. 
      This defines a morphism $\EExt^v_{\cU_e}\to \cP_{\cU_e}$. As in \cite{FM}, proof of Theorem 1.1, this morphism and $\bar\beta$ are inverse to each other.
\end{proof}

\begin{proposition}\label{isot2}  $V\otimes H^1(E^*)\to H^1(E\otimes E^*)$ is an isotropic vector space.
\end{proposition}
\begin{proof} It is analogous to the proof of Proposition \ref{isot1}.
\end{proof}

Fix $(F,V^*)\in \cP_{U_e}$. Our next goal is to construct in an explicit, natural way the linear maps in the exact sequence \[
0\to T_{[(F,V^*)]}(\pi ^{-1}(F))\to  T_{[(F , V^*)]}\cP_{U_e}\to T_{[F]}U_e\to 0 
\]
induced by the smooth morphism $p:\cP_{U_e}\to U_e$. 

Recall that we have identifications $T_{[(F,V^*)]}(\pi ^{-1}(F))  \cong  V\otimes H^1(E^*)$,  $T_{[(F , V^*)]}\cP_{U_e}  \cong  H^1(E\otimes F^*)/V^*\otimes V$ and $T_{[F]}U_e  \cong  H^1(F \otimes F^*)$.

To this end, we consider the exact sequences:
$$
0 \to V^*\otimes V\otimes \cO_X \to V^*\otimes E \to V^*\otimes F \to 0$$
and
$$
0\to F^*\otimes F \to E^*\otimes F  \to V^*\otimes F \to 0;
$$
and we define $G^*$ to be the fiber product of $V^*\otimes E$ and $E^*\otimes F$ over $V^*\otimes F$. Therefore, we have the commutative diagram with exact rows and columns:
$$
 \xymatrix{ &  &0  \ar[d] & 0 \ar[d] \\
   &  &F^*\otimes F  \ar@{=}[r]\ar[d] & F\otimes F^* \ar[d] \\
0 \ar[r] &  V^*\otimes V\otimes \cO_X  \ar[r] \ar@{=}[d] &  G^*  \ar[r] \ar[d] & E^*\otimes F  \ar[r] \ar[d] & 0 \\
0 \ar[r] & V^*\otimes V\otimes \cO_X   \ar[r]  & V^*\otimes E  \ar[r] \ar[d] & V^*\otimes F  \ar[r] \ar[d] & 0 \\
&  & 0 & 0
}$$
Dualizing the second column and the second row we have two short exact sequences:
$$
0\to E^*\otimes V \to G \to F\otimes F^*\to 0 
$$
and
$$ 0\to E\otimes F^*\to G\to V^*\otimes V\otimes \cO_X\to 0.
$$
Hence, natural maps  $$ V\otimes H^1(E^*) \to H^1(G)= H^1(E\otimes F^*)/V^*\otimes V$$ and 
$$H^1(E\otimes F^*)/V^*\otimes V =H^1(G)\to   H^1(F \otimes F^*)$$ as we wanted. 
\vskip 4mm 
Note that again we have a commutative diagram:

$$\xymatrix{H^1(F^*\otimes E)\ar[rr] \ar[dd] \ar[rd] & &    H^1(F^*\otimes F)\\
&  H^1(F^*\otimes E)/V\otimes V^* \ar[ru] \ar@{^(->}[ld] & \\
H^1(E^*\otimes E)  &  & &}$$

\begin{proposition} \label{restrict2-formsbis}  With the above notation we have that the pullback of the closed non-degenerate holomorphic 2-forms on $H^1(F^*\otimes F)$ and $H^1(E^*\otimes E)$   to $T_{[(F , V)]}\cP_{U_e} \cong H^1(F^*\otimes E)/V^*\otimes V$ are the same.\end{proposition}

\begin{proof} 
 It is enough to prove that the  pullback of the symplectic  2-forms on $H^1(F^*\otimes F)$ and $H^1(E^*\otimes E)$ to $H^1(F^*\otimes E)$ are the same.
This is a special case of the following lemma. 
 \end{proof}

\begin{lemma}\label{eppur_commuta}
    Let $Y$ be a ringed space, $E$ and $F$ sheaves of $\mathcal O_Y$-modules which are locally free of finite rank. Fix a morphism $\varphi:E\to F$, i.e. $\phi\in E^*\otimes F$; it induces by composition natural morphisms $E\otimes F^*\to F\otimes F^*$ and $E\otimes F^*\to E\otimes E^*$.
    Then the diagram $$\xymatrix{
    (E\otimes F^*)\otimes (E\otimes F^*)\ar[r] \ar[d] &
    (F\otimes F^*)\otimes (F\otimes F^*)\ar[r]^{\ \ \ \ \ \ \ comp}& \
    F\otimes F^* \ar[d]^{tr}\\  
    (E\otimes E^*)\otimes (E\otimes E^*)\ar[r]_{\ \ \ \ \ \ comp} & 
    E\otimes E^* \ar[r]_{tr} & \mathcal O_Y
}$$
commutes.
\end{lemma}
\begin{proof} To check that the diagram commutes, it is enough to do so on the stalk at every point. So we can assume that $Y$ is one point with ring $A$ and $E$  and $F$ are free $A$-modules of finite rank $m$ and $n$ respectively.\\
After choosing a basis for each, we can identify $\phi$ with an $(n\times m)$ matrix $B$ with coefficients in $A$. 
Since a morphism  $(E\otimes F^*)\otimes (E\otimes F^*)\to A$ is the same as an $A$-bilinear map  $(E\otimes F^*)\times (E\otimes F^*)\to A,$ it is enough to check the statement for two maps $F\to E$, given by $(m\times n)$ matrices $M,N$ with coefficients in $A$.

We must prove that \[
\tr (M\circ B)\circ (N\circ B)= \tr (B\circ M)\circ (B\circ N) 
\]
which follows by associativity of composition and the fact that $\tr(MC)=\tr(CM)$ whenever $C$ is an $(n\times m)$ matrix.
\end{proof}

Putting all together and arguing as in the proof of Theorem:

\begin{theorem}\label{pfthm2}
The smooth, irreducible, locally closed subspace $\cL\subset \Spl(r;c_1,c_2)$ is 
 isotropic (resp.~Lagrangian) if and only if $$\beta _{| \cP_{\cL}}:\cP_\cL\longrightarrow \Spl(\overline{r};c_1,c_2)$$ is an isotropic (resp.~Lagrangian) locally closed embedding. 
\end{theorem}
\begin{proof}
    We first prove the isotropic case. By Proposition \ref{restrict2-formsbis}, $\beta _{| \cP_{\cL}}$ is isotropic at a point $(F,V^*)\in \cP_\cL $ if and only if the image of its tangent space via the differential of the projection $\cP_{\cL} \to \Spl(r;c_1,c_2)$ is an isotropic subspace. That image is $T_{[F]}\cL$ since the projection is a Grassmannian bundle, and thus a submersion.

    Thus the pullback of the symplectic form on $\Spl(\overline{r};c_1,c_2)$ is zero at every point (i.e., $\beta_{\cL}$ is an isotropic immersion) if and only if $T_{[F]}\cL$ is isotropic for every $[F]\in \cL$ (i.e., $\cL$ is an isotropic subspace).

    The Lagrangian statement follows by combining the isotropic case with the dimension calculations in Proposition \ref{dim2}
\end{proof}

\begin{corollary}
    For any $\cL\subset \Spl(r;c_1,c_2)$ locally closed subspace, the locally closed subspace $\Ext^v_{\cL}\subset \Spl(v+r,c_1,c_2)$ is isotropic (respectively Lagrangian) if and only if $\cL$ is isotropic (respectively Lagrangian).
\end{corollary}


\section{On the stability of generalized syzygy bundles}

In this section we will study the stability of the generalized syzygy bundles $S$ associated to a couple $(F,W)$. As we said in section 2, to establish the stability of a generalized syzygy bundle is an open question  of great interest closely related to classical problems in geometry and algebra, ranging from the $(N_p)$ properties in the sense of Green \cite{Gr} to questions of tight closure \cite{B}. When $W=H^0(F)$ asymptotically results are known while the situation drastically change when $W\subset H^0(F)$ (see, for instance, \cite{C}, \cite{CMM}, \cite{MM} and \cite{MS}). We will illustrate by means of an example  that the results strongly depend on the basis that we choose in $W$ and not only on its dimension. We will end this section with a question which naturally arises in this more general context.

\vskip 2mm
Let us first recall the definition and some key result about the (slope) stability of vector bundles.

\begin{definition}\ Let $(X,L)$ be a polarized smooth variety of dimension $d$. A vector bundle $E$ on $X$ is {\em $L-$stable} if for any subsheaf $F\subset E$ with $0<rk(F)<rk(E)$, we have 
\[
\mu_{L}(F):=\frac{c_1(F)L^{d-1}}{rk(F)}<\mu_{L}(E):=\frac{c_1(E)L^{d-1}}{rk(E)}. \]
\end{definition}

\vskip 2mm
The following result is a cohomological characterization of stability, and it will play a central role in the proof of our results in this section.

\begin{lemma}  \label{key} \cite[Lemma 2.1]{C} Let $(X,L)$ be a polarized smooth variety of dimension $d$. Let  $E$ be a vector bundle on $X$. Suppose  that for any integer $q$ and any line bundle $F $ on $X$ such that
$$ 0<q<rk(E) \quad \text{ and } \quad (F\cdot L^{d-1})\ge q\mu _L(E) $$
one has
$ H ^0(X,\bigwedge^q E\otimes F^{\vee})=0.
$
Then, $E$ is $L-$stable.
\end{lemma}
A vector bundle $E$ satisfying the hypothesis of Lemma \ref{key} is said to be {\em cohomologically stable}. It is worthwhile to point out that any cohomological stable vector bundle on a polarized variety  $(X,L)$ is $L$-stable but not vice versa.

\vskip 2mm
For sake of completeness let us recall the asymptotic results about the stability of syzygy bundles associated to complete systems known so far.

\begin{proposition} \label{assymptoticresult} Let $X$ be a K3 surface and $L$ an ample line bundle on $X$. It holds:
\begin{itemize}
    \item[(1)] There is an integer $n_0\gg0$ such that for any integer $m\ge n_0$ the syzygy bundle $M_{mL}$ associated to $(\cO_C(mL),H^0(\cO_X(mL)))$ is $L$-stable; and 
    \item[(2)]  For any  $L$-stable vector bundle $F$ on $X$ of rank $r\ge 2$, there is an integer $n_0\gg 0$ such that for all integer $m\ge n_0$ the syzygy bundle $S$ associated to $(F(m), H^0(F(m))$ is $L$-stable.
\end{itemize}
\end{proposition}
\begin{proof} (1) See \cite[Theorem A]{ELM}.

(2) It follows from \cite[Theorem 1.1]{BP}.
\end{proof}

The following example shows that for non-complete systems (i.e. $W\subset H^0(\cF)$) the  problem is much more subtle.

\begin{example} Let $X\subset \PP^3$ be a smooth  quartic surface with $Pic(X)\cong \ZZ \cong \langle \cO_X(1) \rangle $. Set $L:=\cO_X(7)$, $L':=\cO_X(1)$, and fix homogeneous coordinates in $\PP^3$: $x,y,z,t$. We consider two subspaces $$W_1,W_2\subset H^0(\cO_X(7))\subset H^0(\PP^3,\cO_{\PP^3}(7))=\langle x^iy^jz^nt^m/i+j+n+m=7 \rangle $$
where 
$$ W_1=\langle x^{7}, y^{7}, z^{7},t^{7},x^6y \rangle \text{ and }
$$
$$ W_2=\langle x^{7}, y^{7}, z^{7},t^{7},x^2y^2z^2t \rangle .
$$
The syzygy bundles $S _i$ associated to $(\cO_X(7),W_i)$, $i=1,2$ sit in the short exact sequences
\begin{equation}\label{w1w2}
0 \to S_i \to W_i\otimes \cO_X \to \cO_X(7) \to 0 \quad \text{ for } i=1,2.
\end{equation}
Therefore, we have $\rank(S _i)=4$, $c_1(S_i)=-7L'$ and $\mu _{L'}(S _i)=\frac{-7L'^2}{4}=-7$. Moreover, by Lemma \ref{simple}, $S_1$ and $S_2$ are simple.

\noindent{\em Claim:} $S_1$ is not $L'$-stable while $S_2$ is $L'$-stable.

\noindent{\em Proof of the Claim:} Since the generators of $W_1$ have a linear syzygy, we have $H^0(S_1(L'))\ne 0$ and we get a line subbundle $\cO_X(-1) \subset S_1$ with $\mu _{L'}(\cO_X(-1))=-L'^2=-4>-7=\mu _{L'}(S _1)$ which implies that $S_1$ is not $L'$-stable.

On the other hand, using the exact sequence (\ref{w1w2}), its exterior powers,  the equalities 
$$ \begin{array}{rcl} \rank (\bigwedge^qS_i) & = & \binom{\rank(S_i)}{q} \\
\mu _{L'}(\bigwedge^qS_i) & = & q \mu _{L'}(S _i),
\end{array}
$$
in order to prove that $S_2$ is cohomologically stable (and hence $L'-$stable by Lemma \ref{key}) it is enough to verify that
\[
\begin{array}{ccc}
H^0(S_2(L'))) & = & 0 \\
H^0((\bigwedge^2 S_2(3L'))) & = & 0 \\
H^0((\bigwedge^3S_2(5L'))) & = & 0.
\end{array}
\]
This is a straightforward explicit calculation which we omit.
\end{example}

Since $L-$stability is an open condition, once we know the existence of a subspace $W\subset H^0(\cO_X(a))$ such that the associated syzygy bundle is $L-$stable, then the same works for any subspace $W$ in an open subset of $Gr(w,H^0(\cO_X(a)))$.
A series of examples where this holds is described in the following Proposition.

\begin{proposition} Let $X\subset  \PP^3$ be a smooth quartic surface with $Pic(X)\cong \ZZ \cong \langle \cO_X(1) \rangle $ and set  $L:=\cO_X(1)$. For any integers $2\le a\le 3$ and  $3\le w \le \binom{a+3}{3}$  there is a subspace $W\subset H^0(\cO_X(a))$  of dimension $w$ such that the syzygy bundle $S $ associated to $(\cO_X(a),W)$  is $L$-stable.

\end{proposition}
\begin{proof} First of all we observe that for $2\le a\le 3$, $H^0(\cO_{\PP^3}(a))\cong H^0(\cO_X(a))$ and we will analyze separately the case $w=3$ and $w>3$.

For $w=3$ we take $W=\langle f_1,f_2,f_3\rangle \subset H^0(\cO_{\PP^3}(a))\cong H^0(\cO_X(a))$ being $f_i$ three forms of degree $a$ without linear syzygies. We construct the rank 2 syzygy bundle $S$ on $X$ associated to $(\cO_X(a),W)$:
$$
0 \to S \to W\otimes \cO_X \to \cO_X(a)\to 0
$$
verifying $H^0(S(1))=0$. Applying Lemma \ref{key} we get that $S$ is $L$-stable.

Assume $w>3$. By \cite[Theorem 4.6]{MM},  there is a subspace $W\subset H^0(\cO_{\PP^3}(a))\cong H^0(\cO_X(a))$ such that the syzygy bundle $\Sigma $ on $\PP^3$ associated to $(\cO_{\PP^3}(a),W)$ is cohomologically stable. Call $S$  the syzygy bundle $\Sigma $ on $X$ associated to $(\cO_X(a),W)$. We clearly have that $S$ is isomorphic to the restriction of $\Sigma $ to $X$, i.e. $S\cong \Sigma \otimes \cO_X$ and, for $1\le q\le w-1$, we have the commutative diagram:
$$
 \xymatrix{ & 0 \ar[d] &0  \ar[d] & 0 \ar[d] \\
0 \ar[r] &  \bigwedge^q\Sigma (-4)  \ar[r] \ar[d] &  \bigwedge^qW\otimes \cO_{\PP^3}(-4)  \ar[r] \ar[d] & \bigwedge^{q-1}\Sigma (a-4)  \ar[r] \ar[d] & 0 \\
0 \ar[r] & \bigwedge^q\Sigma   \ar[r] \ar[d] & \bigwedge^q W\otimes \cO_{\PP^3}  \ar[r] \ar[d] & \bigwedge^{q-1}\Sigma (a)  \ar[r] \ar[d] & 0 \\
0 \ar[r] & \bigwedge^qS  \ar[r] \ar[d] &  \bigwedge^qW\otimes \cO_{X}  \ar[r] \ar[d] & \bigwedge^{q-1}S(a)  \ar[r] \ar[d] & 0 \\
& 0 & 0 & 0
}$$
A straightforward computation using induction on $q$, the cohomological stability of $\Sigma $ and the exact cohomology sequences coming from the above commutative diagram gives us that $S$ is cohomologically stable which proves what we want.
\end{proof}
We end the paper with a question which naturally arises from our work.

\begin{question}
Let us fix $L$ an ample divisor on a K3 surface and a rank $r$, $L$-stable vector bundle $F$ on $X$. For which integers $w$ with $2+r\le w\le \dim H^0(X,F)$ is there a $w-$dimensional subspace $W\subset H^0(F)$ such that $W\otimes \cO_X \twoheadrightarrow F$ and the syzygy bundle $S$ is $L$-stable?
\end{question}

Again  the existence of a vector subspace $W\subset H^0(F)$ such that $W\otimes \cO_X \twoheadrightarrow F$ and the syzygy bundle $S$ is $L$-stable implies the existence of an open subset in $Gr(w,H^0(F))$ where the same holds.

\vskip 2mm
Finally it is worthwhile to point out that analogous phenomena appear for extension bundles.


\end{document}